\documentclass[12pt]{amsart}
\usepackage[margin=1.19in]{geometry}
\usepackage{amsmath,amsthm,amssymb,mathtools,amsfonts,mathrsfs}
\usepackage{xspace,xcolor}
\usepackage[breaklinks,colorlinks,citecolor=teal,linkcolor=teal,urlcolor=teal,pagebackref,hyperindex]{hyperref}
\usepackage[alphabetic]{amsrefs}
\usepackage{tikz-cd}
\usepackage[all]{xy}

\newtheorem{thm}{Theorem}[section]
\newtheorem{thmx}{Theorem}[section]

\newtheorem{prop}[thm]{Proposition}

\newtheorem{lem}[thm]{Lemma}

\theoremstyle{definition}
\newtheorem{defn}[thm]{Definition}
\newtheorem{ex}[thm]{Example}
\newtheorem{rmk}[thm]{Remark}

\newtheorem{assu}{Assumption}
\newtheorem{conv}{Convention}
\newtheorem*{conv*}{Convention}
\newtheorem*{nota*}{Notation}
\newtheorem*{fact*}{Fact}
\newtheorem*{ind*}{Induction hypothesis}
\newtheorem*{goal*}{Goal}

\newtheorem*{ques*}{Question}
\newtheorem{ques}{Question}

\newtheorem*{consq*}{Consequence}

\newcommand{\C}{\mathbb C}

\newcommand{\Z}{\mathbb Z}
\newcommand{\N}{\mathbb N}

\newcommand{\Pp}{\mathbb P}
\newcommand{\X}{\mathfrak X}

\newcommand{\sO}{\mathcal O}
\newcommand{\NE}{NE}

\newcommand{\bM}{\overline{\mathcal M}}
\newcommand{\FF}{\mathfrak F}

\newcommand{\comp}{\mathfrak c}

\title{Chern classes and Gromov--Witten theory of projective bundles}

\author[H.~Fan]{Honglu Fan}
\email{honglu.fan@math.ethz.ch}

\begin{document}

\begin{abstract}
We prove that the Gromov--Witten theory (GWT) of a projective bundle can be determined by the Chern classes and the GWT of the base. It completely answers a question raised in a previous paper. Its consequences include that the GWT of the blow-up of X at a smooth subvariety Z is uniquely determined by GWT of X, Z plus some topological data.
\end{abstract}

\maketitle
\tableofcontents

\setcounter{section}{-1}

\section{Introduction}

In \cite{FL}, the following question was raised.
\begin{ques*}
Let $S$ be a smooth variety and $V$ be a vector bundle. Is the Gromov--Witten invariants of $\Pp_S(V)$ uniquely determined by that of the base $S$ and the total Chern class $c(V)$?
\end{ques*}

The main result of this paper is an affirmative answer to the above question in all genera.
\begin{thmx}[=Theorem \ref{main}]\label{thmA}
Let $S$ be a smooth projective variety, and $V_1,V_2$ be two vector bundles on $S$. If $c(V_1)=c(V_2)$, we have the following equality between Gromov--Witten invariants
\begin{equation}
\langle\psi^{k_1}\alpha_1,\ldots{},\psi^{k_n}\alpha_n\rangle_{g,n,\beta}^{\Pp(V_1)}=\langle\psi^{k_1}\FF(\alpha_1),\ldots{},\psi^{k_n}\FF(\alpha_n)\rangle_{g,n,\Psi(\beta)}^{\Pp(V_2)}
\end{equation}
for any cohomology classes $\alpha_1,\ldots{},\alpha_n\in H^*(X_1)$, any set of natural numbers $k_1,\ldots{}, k_n$, any curve class $\beta\in N_1(X_1)$ and any genus  $g\in \Z_{\geq 0}$.
\end{thmx}

Here $\FF, \Psi$ are natural isomorphisms identifying $H^*(\Pp(V_1))$ with $H^*(\Pp(V_2))$, and $N_1(\Pp(V_1))$ with $N_1(\Pp(V_2))$, respectively (see section \ref{section:main}).

\subsection{Background}
As a natural construction in geometry, the projectivization of vector bundles (projective bundles) has appeared as key objects in various mathematical theories. In particular, it's important in Gromov--Witten theory as well. For example, Gromov--Witten theory of projective bundles has appeared as a central subject in crepant transformation conjecture for ordinary flops (e.g., \cite{LLW1, LLW2, LLQW}), and the later proposals of the \emph{functoriality} of Gromov--Witten theory (\cite{LLW-string, LLW-ams}) also emphasize the role of projective bundles. Besides, projective bundle appears naturally in the degeneration to the normal cone which is an important construction in some Gromov--Witten results (e.g., \cite{MP,HLR}).

\subsection{Techniques and difficulties}
In \cite{MP, GB, CGT}, the virtual localization under the fiberwise torus action has been very effective in the study of Gromov--Witten theory of split projective bundles. In spite of this, it is difficult to establish Theorem \ref{thmA} in the case of split bundles using localization. The main hurdle is that localization would not yield symmetric expressions in Chern roots until nonequivariant limits are taken. A way to see this is probably the following. For simplicity suppose we are given two split rank-$2$ bundles $V_1=L_1\oplus L_2$ and $V_2=L_1'\oplus L_2'$ with $c(V_1)=c(V_2)$ but $c_1(L_1)\neq c_1(L_i')$ for both $i=1,2$. If we impose torus actions on $V_1$ and $V_2$ by scaling each factor, one can compute that $c_{T}(V_1)\neq c_{T}(V_2)$. Therefore, we cannot expect $\Pp(V_1)$ and $\Pp(V_2)$ to have the same equivariant theory, and Theorem \ref{thmA} only holds when passing to their nonequivariant limits. One might notice that the I-function in \cite{GB} packages the localization computation into a form that admits a nonequivariant limit. But unfortunately the I-function in \cite{GB} is not symmetric in Chern roots either.

When the projective bundle is not split, such a fiberwise torus action is absent. A degeneration argument is applied in \cite{LLQW} and has had some success in a limited form of \emph{quantum splitting principle}.

In this paper, Theorem \ref{thmA} is established using a new method which works regardless of the splitting of a projective bundle. The idea is the following. We compactify a vector bundle $V$ into $\Pp(V\oplus \sO)$. The scaling action on $V$ by $\C^*$ can be extended into $\Pp(V\oplus \sO)$ and we use $\bM_{g,n}(\Pp(V\oplus \sO),\beta)$ as a master space. Using virtual localization, we obtain relations between (twisted) Gromov--Witten invariants of $S$ and $\Pp(V)$ by integrating a cohomology class of degree strictly less than the virtual dimension. By using this kind of relations, we determine Gromov--Witten theory of $\Pp(V)$ out of the one of $S$, and the whole algorithm only involves Chern classes of $V$.

\subsection{Consequences and remarks}
The result is surprising to us because Chern classes are merely topological invariants which are not enough to classify topological types of vector bundles. For example,
\begin{fact*}[\cite{AR}]
In $\Pp^3$, suppose $V$ is a rank $2$ holomorphic vector bundle with even first Chern class. There exists a holomorphic vector bundle $V'$ such that $c_1(V)=c_1(V'), c_2(V)=c_2(V')$ while $V\not\cong V'$ as complex vector bundles.
\end{fact*}


In the meantime, Theorem \ref{main} has an interesting consequence. In section \ref{section:appl}, we explain that our result immediately applies to \cite{HLR} to get the following refinement of \cite[Lemma 1]{MP}.
\begin{thmx}[=Theorem \ref{blowup}]
Let $Z\subset Y$ be inclusions of smooth projective varieties and $N_{Z/Y}$ be the normal bundle of $Z$. Let $\tilde Y$ be the blow-up of $Y$ at $Z$ and $E$ be the exceptional divisor. The absolute Gromov--Witten invariants of $\tilde Y$ can be determined by the absolute Gromov--Witten invariants of $Y$ and $Z$, plus the following topological data:
\begin{enumerate}
\item The cohomology rings $H^*(Y)$, $H^*(Z)$ and their pull-back maps under inclusion.
\item The Chern classes $c_i(N_{Z/Y})\in H^*(Y)$.
\end{enumerate}
\end{thmx}

All the above results motivate us to ask some further questions. For example,
\begin{ques}
Can we establish a similar theorem for arbitrary $\Pp^n$-fibrations (not necessarily the projectivization of a vector bundle)?
\end{ques}

More precisely, a $\Pp^n$-fibration (smooth fibration with each fiber $\Pp^n$) will still satisfy the condition of Leray--Hirsch theorem, and we can impose a condition by requiring that the identification of the first Chern class of relative canonical sheaves induces an isomorphism between cohomology rings of given fibrations. One can see that this condition generalizes the one in Theorem \ref{thmA}. On the other hand, a $\Pp^n$-fibration admits other invariants, for example, the Brauer class. It might be interesting to study whether Brauer class or other invariants play a role here. 

Regarding the blow-up theorem, it would be interesting to see whether the whole algorithm can be made explicit. A further question we have in mind is the following.


\begin{ques}\label{q2}
Is it possible to state a relatively ``nice" relationship (say, using symplectic operators depending on $c_i(N_{Z/X})$) between Lagrangian cones $\mathcal L_{Bl_ZX}$, $\mathcal L_X$ and $\mathcal L_Z$? 
\end{ques}

\subsection{Acknowledgement}
The author would like to thank his advisor Y. P. Lee for his guidance and encouragement throughout the project. He would also like to thank Andrei Musta\c{t}\u{a}, Junliang Shen, Yang Zhou and  for helpful discussions. The author is supported by NSF and SwissMAP.

\section{Statement of the main theorem}\label{section:main}

In this paper, $c(V)$ refers to the total Chern class, i.e.
\[
c(V)=1+c_1(V)+\cdots{}+c_r(V).
\]

Let $S$ be a smooth projective variety. $V_1, V_2$ are vector bundles of rank $r$ over $S$. We further require that
\[
c(V_1)=c(V_2).
\]
Let $X_1=\Pp(V_1), X_2=\Pp(V_2)$ be the corresponding projective bundles, and
\[
\pi_i: X_i\rightarrow S
\]
are the projections for $i=1,2$. From Leray--Hirsch we know
\[
H^*(X_i)=H^*(S)[h_i]/(h_i^r+c_1(V_i)h_i^{r-1}+\cdots{}+c_r(V_i)), (i=1,2),
\]
where $h_i=\sO_{X_i}(1)$. Because their Chern classes agree, we have the following isomorphism
\[
\FF_{V_1,V_2}:H^*(X_1)\cong H^*(X_2)
\]
satisfying $\FF_{V_1,V_2}(\pi_1^*\sigma)=\pi_2^*\sigma, ~\FF(h_1)=h_2$. There is also an isomorphism
\[
\Psi_{V_1,V_2}:N_1(X_1)\cong N_1(X_2)
\]
determined by the intersection property
\[
(D,\beta)=(\FF_{V_1,V_2}(D),\Psi_{V_1,V_2}(\beta)),
\]
for any $D\in H^2(X_1)$. To shorten the notation, we adopt the following convention,
\begin{conv}
When the context is clear, we write $\FF$ instead of $\FF_{V_1,V_2}$, and $\Psi$ instead of $\Psi_{V_1,V_2}$.
\end{conv}



Now we state our main theorem.
\begin{thm}\label{main}
Let $S$ be a smooth projective variety, and $V_1,V_2$ be two vector bundles on $S$. Let $X_1, X_2, \FF$ also be the same as above. Suppose $c(V_1)=c(V_2)$. Then we have the following equality between Gromov--Witten invariants.
\begin{equation}\label{eqn:maineqn}
\langle\psi^{k_1}\alpha_1,\ldots{},\psi^{k_n}\alpha_n\rangle_{g,n,\beta}^{X_1}=\langle\psi^{k_1}\FF(\alpha_1),\ldots{},\psi^{k_n}\FF(\alpha_n)\rangle_{g,n,\Psi(\beta)}^{X_2}
\end{equation}

for any cohomology classes $\alpha_1,\ldots{},\alpha_n\in H^*(X_1)$, any set of natural numbers $k_1,\ldots{}, k_n$, any curve class $\beta\in N_1(X_1)$ and any genus  $g\in \Z_{\geq 0}$.
\end{thm}


Let's make a preliminary observation towards the theorem.
\begin{lem}\label{lemma:curve}
$(K_{X_1},\beta)=(K_{X_2},\Psi(\beta))$.
\end{lem}

As a result, the virtual dimensions of the involved moduli spaces of stable maps are the same, which is one of the first things one can check directly for the theorem.

\begin{proof}
Recall the Euler sequence
\[
0\rightarrow \sO_{X_i}\rightarrow \pi_i^*V_i \otimes \sO_{\Pp(V_i)}(1) \rightarrow T_{\pi_i}\rightarrow 0,
\]
where the $T_{\pi_i}$ is the relative tangent sheaf of the morphism $\pi_i$. We also know that for $i=1, 2$
\[
K_{X_i}=K_{\pi_i} \otimes \pi_i^*K_{S}=det(T_{\pi_i}^\vee) \otimes \pi_i^*K_{S}.
\]
Thus, it suffices to prove
\[
\FF(c_1(det(T_{\pi_1}^\vee)))=c_1(det(T_{\pi_2}^\vee)).
\] 
We see from the Euler sequence that 
\[
c(T_{\pi_i})=c(\pi_i^*V_i \otimes \sO_{\Pp(V_i)}(1))
\] 
for $i=1, 2$. The proof can be easily finished with a splitting principle calculation using the fact that $\FF\circ \pi_1^*=\pi_2^*$, $\FF(c_1(\sO_{\Pp(V_1)}(1)))=c_1(\sO_{\Pp(V_2)}(1))$ and $c(V_1)=c(V_2)$.
\end{proof}

Note that in proving the theorem, we are free to twist both $V_1$ and $V_2$ by a fixed line bundle \emph{simultaneously}. To be more precise, for any line bundle $L$ on $S$, we still have
\[
c(V_1\otimes L^{-1})=c(V_2 \otimes L^{-1})
\]
and
\[
\Pp(V_i\otimes L^{-1})\cong \Pp(V_i).
\]
Throughout this paper, we choose $L$ to be sufficiently ample so that $\sO_{\Pp(V_i\otimes L^{-1})}(1)=\sO_{\Pp(V_i)}(1)+\pi_i^*L$ is ample on $X_i$. Therefore,without loss of generality, we make the following assumption in this paper.
\begin{assu}\label{assumption}
We assume that $\sO_{\Pp(V_i)}(1)$ is ample on $\Pp(V_i)$ for $i=1,2$.
\end{assu}

Later, more specific requirements about the ampleness of $L$ will be made.

\section{Twisted Gromov--Witten invariants}
The proof of Theorem \ref{main} roughly goes by relating corresponding invariants on $X_1$ and $X_2$ to the same expression of twisted invariants on $S$. In this section, we set up the notations and recall basic facts about twisted Gromov--Witten invariants.

Let $X$ be a smooth projective variety, $E$ be a vector bundle over $X$. Let 
\[
ft_{n+1}:\bM_{g,n+1}(X,\beta) \rightarrow \bM_{g,n}(X,\beta)
\]
be the map forgetting the last marked point. Under this map $\bM_{g,n+1}(X,\beta)$ can be viewed as the universal family over $\bM_{g,n}(X,\beta)$. Furthermore, the evaluation map of the last marked point
\[
ev_{n+1}:\bM_{g,n+1}(X,\beta) \rightarrow X
\]
serves as the universal stable map from the universal family over $\bM_{g,n}(X,\beta)$.

Let $\C^*$ act on $X$ trivially and on $E$ by scaling (weight $1$ on every $1$-dimensional linear subspace of a fiber). Let $\lambda$ be the corresponding equivariant parameter. There exists a two-term complex of vector bundles in $\bM_{g,n}(X,\beta)$
\[
0\rightarrow E_{g,n,\beta}^0 \rightarrow E_{g,n,\beta}^1\rightarrow 0
\]
such that the $i$-th cohomology is $R^i(ft_{n+1})_*ev_{n+1}^*E$ for $(i=0,1)$.

Let $\C^*$ act on $E_{g,n,\beta}^i$ by scaling as well. Write $E_{g,n,\beta}$ for the two-term complex $[E_{g,n,\beta}^0 \rightarrow E_{g,n,\beta}^1]$ in $D^b(\bM_{g,n}(X,\beta))$. Define the equivariant Euler class 
\[
e_{\C^*}(E_{g,n,\beta})=\displaystyle\frac{e_{\C^*}(E_{g,n,\beta}^0)}{e_{\C^*}(E_{0,n,\beta}^1)} \in H^*(\bM_{g,n}(X,\beta))\otimes_\C \C[\lambda,\lambda^{-1}].
\]
As a result,
\[
e_{\C^*}(E_{g,n,\beta}^1)=\lambda^{r_1}+c_1(E_{g,n,\beta}^1)\lambda^{{r_1}-1}+\cdots{}+c_r(E_{g,n,\beta}^1)\in H^*(\bM_{g,n}(X,\beta))\otimes_{\C}\C[\lambda],
\]
where $r_1=\text{rank}(E_{g,n,\beta}^1)$. Since elements in $H^*(\bM_{0,n}(X,\beta))$ are nilpotent, one easily sees that $e_{\C^*}(E_{g,n,\beta})$ is well-defined if we invert $\lambda$. Applying the same reason on $E^0_{g,n,\beta}$, we see $e_{\C^*}(E_{g,n,\beta})$ is invertible if $\lambda$ can be inverted.

Define the twisted Gromov--Witten invariants to be
\[
\langle \psi^{k_1}\alpha_1,\ldots{},\psi^{k_n}\alpha_n \rangle_{g,n,\beta}^{X,tw,E}=\displaystyle\int_{[\bM_{0,n}(X,\beta)]^{vir}} \frac{1}{e_{\C^*}(E_{g,n,\beta})} \cup \prod\limits_{i=1}^n \psi_i^{k_i}ev_{i}^*\alpha_i \in \C[\lambda,\lambda^{-1}] .
\]
\begin{rmk}
Under the more general framework in \cite{cg}, our twisted invariants are only their special case by choosing the multiplicative characteristic class to be $\displaystyle\frac{1}{e_{\C^*}(\cdot)}$. Throughout our paper, this is the only case we need.
\end{rmk}

\begin{lem}\label{lemma:untw}
If the insertions are homogeneous and satisfy 
\[
\sum\limits_{i=1}^n k_i+\sum\limits_{i=1}^n deg(\alpha_i)=dim([\bM_{g,n}(X,\beta)]^{vir}) ,
\]
then
\[
\langle\psi^{k_1}\alpha_1,\ldots{},\psi^{k_n}\alpha_n\rangle_{g,n,\beta}^{X,tw,E}=\displaystyle\frac{\langle\psi^{k_1}\alpha_1,\ldots{},\psi^{k_n}\alpha_n\rangle_{g,n,\beta}^{X}}{\lambda^r}
\]
for some $r$.
\end{lem}

\begin{proof}
One can expand to see that 
\[
\displaystyle\frac{1}{e_{\C^*}(E_{g,n,\beta})}=\frac{1}{\lambda^r}\left(1+\cdots{} \right),
\]
where $r=\text{rank}(E_{g,n,\beta}^0)-\text{rank}(E_{g,n,\beta}^1)$. Each summand in $\cdots{}$ at the end involves cohomology classes of nonzero degrees. Since the insertion already agrees with the virtual dimension of the moduli space, only the leading term may produce a nonzero number.
\end{proof}


\section{Virtual localization}\label{section:vloc}

In order to carry out the recursive algorithm in Section \ref{section:recursion}, we need to use virtual localization developed in \cite{GP}. We would like to briefly recall their technique and make the set-up in the next subsection.

\subsection{Localization}\label{section:localization}
Let $X$ be a smooth projective variety admitting an action by a torus $T=\C^m$. It induces an action of $T$ on $\bM_{g,n}(X,\beta)$. Let $\bM_\alpha$ be the connected components of the fixed loci $\bM_{g,n}(X,\beta)^T$ labeled by $\alpha$ with the inclusion $i_\alpha: \bM_\alpha\rightarrow \bM_{g,n}(X,\beta)$. The virtual fundamental class $[\bM_{g,n}(X,\beta)]^{vir}$ can be written as
\[
[\bM_{g,n}(X,\beta)]^{vir}=\sum_\alpha (i_\alpha)_! \displaystyle\frac{[\bM_\alpha]^{vir}}{e_T(N^{vir}_\alpha)},
\]
where $[\bM_\alpha]^{vir}$ is constructed from the fixed part of the restriction of the perfect obstruction theory of $\bM_{g,n}(X,\beta)$, and the virtual normal bundle $N^{vir}_\alpha$ is the moving part of the two-term complex in the perfect obstruction theory of $\bM_{g,n}(X,\beta)$. The indices $\alpha$ in the above virtual localization formula can be replaced by decorated graphs introduced later in this section.

We will focus on the case $T=\C^*$. The set-up of our localization computation is the following. Let $S$ be a smooth projective variety and $V$ a vector bundle on $S$. We would like to focus on $\Pp(V\oplus\sO)$. Denote
\[
X=\Pp(V\oplus\sO).
\]
There is a natural inclusion
\[
\Pp(V)\hookrightarrow X
\]
whose complement can be naturally identified as the total space of the vector bundle $V$, i.e.,
\begin{equation}\label{eqn:v}
V\cong X-\Pp(V).
\end{equation}

The $\C^*$ action on $V$ by scaling extends to $X$ in the obvious way. Under this $\C^*$ action, the fixed loci of $X$ are the following.
\begin{enumerate}
\item A copy of $S$ (denoted by $X_0$) which can be identified as the zero section of $V$ under the isomorphism in (\ref{eqn:v}).
\item A copy of $\Pp(V)$ (denoted by $X_\infty$) which can be identified as $X \backslash V$.
\end{enumerate}
We denote $X^T=X_0\bigcup X_\infty$. Their normal bundles are the following.
\begin{enumerate}
\item $N_{X_0/X}=V$. There is an induced fiberwise $\C^*$ action of character $1$ for any sub-representation.
\item $N_{X_\infty/X}=\sO_{\Pp(V)}(1)$. There is an induced fiberwise $\C^*$ action of character $-1$ for any sub-representation.
\end{enumerate}

\begin{prop}\label{proposition:equivcoh}
The equivariant cohomology
\[
H^*_{\C^*}(X)=H^*(S)[\lambda,h]/(h^r+c_1(V)h^{r-1}+\cdots{}+c_r(V))(h-\lambda),
\]
where $r=rank(V)$, $\lambda$ is the equivariant parameter, and $h$ is the equivariant cohomology class that restricts to $\lambda$ on $X_0$ and to $c_1(\sO_{\Pp(V)}(1))$ on $X_\infty$.
\end{prop}
We omit the standard computation.

\subsection{Decorated graph}
Consider an invariant stable map $f:(C,x_1,\ldots{},x_n)\rightarrow X$. We can associate a graph $\Gamma$ (with decorations) to it. Let $V(\Gamma)$ and $E(\Gamma)$ be the set of vertices and edges, respectively. The rule of the assignment is the following.

\medskip\noindent {\bf Vertices}:
\begin{itemize}
\item The connected components in $f^{-1}(X^T)$ are either curves or points. Assign a vertex $v$ to a connected component $\comp_v$ in $f^{-1}(X^T)$.
\item Define $p_v=0$ or $\infty$ depending on whether $f(\comp_v)\subset X_0$ or $X_\infty$, respectively.
\item Define $\beta_v=f_*[\comp_v]\in N_1(X_{p_v})$ the numerical class of $f_*([\comp_v])$ (if $f(\comp_v)$ is a point, $[\comp_v]=0$).
\item Define $g_v$ to be the arithmetic genus of $\comp_v$.
\item For the $i$-th marked point ($i=1,\ldots{},n$), define $s_i=v$ if $x_i\in \comp_v$.
\end{itemize}
\noindent {\bf Edges}:
\begin{itemize}
\item Assign each component of $C-\mathop{\bigcup}\limits_{v\in V(\Gamma)}\comp_v$ an edge $e$. Let $\comp_e$ be the closure of the corresponding component.
\item We write $\beta_e=f_*[\comp_e]\in N_1(X)$.
\end{itemize}

We also introduce the following.
\begin{nota*}
\begin{itemize}
\item $\beta_\Gamma=\sum\limits_{e\in E(\Gamma)}\beta_e+\sum\limits_{v\in V(\Gamma)}\beta_v$
\item $g_\Gamma=\sum\limits_{v\in V(\Gamma)} g_v+h^1(\Gamma)$ where $h^1(\Gamma)$ is the dimension of first singular cohomology of $\Gamma$ as a one-dimensional CW complex.
\item $n_v$ is the number of markings on the component $\comp_v$.
\end{itemize}
\end{nota*}

The graph assignment is invariant under deformations of invariant stable maps. Thus, we label each connected component of fixed loci of $\bM_{g,n}(X,\beta)$ by the decorated graphs above. Denote the component by $\bM_\Gamma$ and the virtual normal bundle $N_\Gamma^{vir}$. Now one can construct the fixed component as certain fiber product according to the decorated graph, and write down $\displaystyle\frac{1}{e(N_\Gamma)^{vir}}$ more explicitly. The general framework is treated in \cite{MM}.

In order to write the virtual localization formula, we need to introduce a few more definitions (following \cite[Definition 53]{liu}).
\begin{defn}\label{defn:stable}
A vertex $v\in V(\Gamma)$ is called stable if $2g_v-2+val(v)+n_v>0$ or $\beta_v\neq 0$. Let $V^S(\Gamma)$ be the set of stable vertices in $V(\Gamma)$. Let
\begin{align*}
V^1(\Gamma) &= \{ v\in V(\Gamma) ~|~ g_v=0, val(v)=1, n_v=0, \beta_v=0 \}, \\
V^{1,1}(\Gamma) &= \{ v\in V(\Gamma) ~|~ g_v=0, val(v)=n_v=1, \beta_v=0 \}, \\
V^{2}(\Gamma) &= \{ v\in V(\Gamma) ~|~ g_v=0, val(v)=2, n_v=0, \beta_v=0 \}.
\end{align*}
The union of $V^1(\Gamma), V^{1,1}(\Gamma), V^2(\Gamma)$ is the set of unstable vertices.
\end{defn}

Define an equivalence relation $\sim$ on the set $E(\Gamma)$ by setting $e_1\sim e_2$ if there is a $v\in V^2(\Gamma)$ such that $e_1, e_2\in E_v$.
\begin{defn}
Define $\overline E(\Gamma)=E/\sim$.
\end{defn}

One easily sees that a class $[e]\in \overline E(\Gamma)$ consists of a chain of edges, say $e_1,e_2,\ldots{},e_m$ such that $e_i$ and $e_{i+1}$ intersect at a $v_i\in V^2(\Gamma)$. There are also two vertices $v_0\in e_1$ and $v_m\in e_m$ such that $v_0, v_m\not\in V^2(\Gamma)$.
\begin{defn}\label{defn:vun}
Define $V^{2}_{[e]}=\{v_1,\ldots{},v_{m-1}\}$ and $V^{end}_{[e]}=\{v_0, v_m\}$.
\end{defn}

\begin{defn}\label{defn:tail}
Define $\overline E^{tail}(\Gamma)$ be the set of edge classes $[e]\in \overline E(\Gamma)$ such that $V^{end}_{[e]}\bigcap V^1(\Gamma)\neq \emptyset$ or $V^{end}_{[e]}\bigcap V^{1,1}(\Gamma)\neq \emptyset$.
\end{defn}

\begin{defn}\label{defn:sides}
Define $V^{\infty}(\Gamma)=\{v\in V(\Gamma) ~|~ p_v=\infty\}$ and $V^{0}(\Gamma)=\{v\in V(\Gamma) ~|~ p_v=0\}$.
\end{defn}

Definitions \ref{defn:stable}-\ref{defn:sides} are used to describe some summation index in the virtual localization formula.

\section{Computing invariants on a projective bundle}\label{section:recursion}

Let $S$ be the smooth projective variety. In this section, we focus on a single vector bundle $V$. Write $\pi:\Pp(V)\rightarrow S$ to be the projection. Later in this section, we recursively establish an algorithm genus by genus. For each genus $g$, there are some additional data we need to choose before the recursion begins.
\begin{lem}
For any $g\in \Z_{\geq 0}$, there is a sufficiently ample line bundle $L_g\in \text{Pic}(S)$ such that for any $\beta\in \NE(\Pp(V))$ with $\pi_*\beta\neq 0$, we have
\[
( \beta , \sO_{\Pp(V)}(1)+\pi^*L_g ) > \text{max}\{g-1,0\},
\]
(where $\NE(\cdots{})$ means the Mori cone).
\end{lem}
\begin{proof}
For an ample $L$, we always have $(\overline\beta,L)\geq 1$ for any nonzero $\overline\beta \in \NE(S)$. By replacing $L$ by its multiple $mL$, we can assume that $(\overline\beta,L)\geq m$ holds for any integer $m$ and any nonzero $\overline\beta \in \NE(S)$. Back to the lemma, for any $\beta \in \NE(\Pp(V))$, notice that
\[
( \beta , \sO_{\Pp(V)}(1)+\pi^*L_g ) = ( \beta , \sO_{\Pp(V)}(1) ) + ( \beta , \pi^*L_g ) = ( \beta , \sO_{\Pp(V)}(1) ) + ( \pi_*\beta , L_g ),
\]
where the last equality follows from the projection formula and the last intersection pairing $( \pi_*\beta , L_g )$ is evaluated in $S$. Since $g-1$ is a fixed integer, the inequality can be achieved by choosing $L$ to be a sufficiently high multiple of an ample line bundle.
\end{proof}

Now that the collection of ample line bundles $\{L_g\}_{g\in \N}$ are chosen, we can state the following theorem.
\begin{thm}\label{recursive}
Let $f\in N_1(\Pp(V))$ be the class of a line in a fiber. Suppose the fiber integrals $\langle \ldots{} \rangle_{g,n,kf}^{\Pp(V)}$ are known. There is an algorithm determining genus $g$ untwisted invariants of $\Pp(V)$ from the genus $g'$ twisted invariants of $S$ (twisted by $V\otimes L_{g''}$ where $g''\leq g$) with $g'\leq g$. Furthermore, besides the twisted invariants of $S$, this algorithm only depends on the cohomology rings $H^*(S)$ and $H^*(\Pp(V))$, the cohomology class $c_1(\sO_{\Pp(V)}(1))$, the group of numerical curve classes $N_1(\Pp(V))$ with its intersection pairing with $H^2(\Pp(V))$, the pull-back morphism $\pi^*:H^*(S)\rightarrow H^*(\Pp(V))$, and the Mori cone $\NE(\Pp(V))$.
\end{thm}

We rule out the fiber integrals $\langle \ldots{} \rangle_{g,n,kf}^{\Pp(V)}$ because they serve as the initial case for an induction, obtained by a method different from localization. We summarize the structure of this section as follows. We study fiber integrals in Section \ref{section:fiber}. Next, we prove Theorem \ref{recursive} by introducing the set-up in sections \ref{section:setup}, determining insertions to be used in the master space in \ref{section:lifting}, and writing out the computation in \ref{section:loc}.

\subsection{Fiber classes}\label{section:fiber}

In this section, we determine integrals of the form $\langle \ldots{} \rangle_{g,n,kf}^{\Pp(V)}$ where $f\in N_1(\Pp(V))$ is the line class in a fiber. Similar to  \cite[section 1.2]{MP}, $\bM_{g,n}(\Pp(V),kf)$ is fibered over $S$, i.e. there is a morphism
\[
p:\bM_{g,n}(\Pp(V),kf) \rightarrow S.
\]
In fact, $\bM_{g,n}(\Pp(V),kf)=\mathcal P\times \bM_{g,n}(\Pp^r,k)/GL(r+1)$ where $\mathcal P$ is the principal $GL(r+1)$ bundle corresponding to $V$, and $GL(r+1)$ acts on the product diagonally. We can decompose the virtual class as
\[
[\bM_{g,n}(\Pp(V),kf)]^{vir}=e(\mathbb E \boxtimes T_S)\cap [\bM_{g,n}(\Pp(V),kf)]^{vir_p},
\]
where $\mathbb E$ is the Hodge bundle and $[\bM_{g,n}(\Pp(V),kf)]^{vir_p}$ is the relative virtual class. This reduces our problem to the determination of integrals against $[\bM_{g,n}(\Pp(V),kf)]^{vir_p}$. This relative virtual class can be described in terms of something that is well-understood.

Without breaking Assumption \ref{assumption}, we can replace $V$ by $V\otimes L^{-1}$ where $L$ is sufficiently ample. As a result, we can assume $V^\vee$ is globally generated and there is a surjection map $\sO^N\rightarrow V^\vee$ for some large integer $N$. By taking the dual, we embed $V$ into a trivial bundle
\[
V\hookrightarrow \sO^N.
\]
By the universal property of Grassmannian, it induces a morphism
\[
f:S\rightarrow Gr(r+1,N)
\]
such that $f^*U=V$ where $U$ is the tautological bundle of rank $r+1$. We write $G=Gr(r+1,N)$ in short. In the meantime, we have the $\bM_{g,n}(\Pp_G(U),kf)$ whose base change via $f$ is isomorphic to $\bM_{g,n}(\Pp(V),kf)$. Now we have the diagram
\[\xymatrix{
\bM_{g,n}(\Pp_S(V),kf) \ar[d]^p \ar[r]^{\bar f} & \bM_{g,n}(\Pp_G(U),kf) \ar[d]^{\bar p} \\
S \ar[r]^f & G
}
\]

One sees that $\bar f^![\bM_{g,n}(\Pp_G(U),kf)]^{vir_{\bar p}}=[\bM_{g,n}(\Pp_S(V),kf)]^{vir_p}$. As a result, 
\[
\bar p^* f_! \sigma \cap [\bM_{g,n}(\Pp_G(U),kf)]^{vir_{\bar p}}=\bar p^*\sigma \cap \bar f_! [\bM_{g,n}(\Pp_S(V),kf)]^{vir_p},
\]
where $\sigma\in H^*(S)$.

An insertion is in the form $h^i\pi^*\sigma\in H^*(\Pp_S(V))$ where $h=c_1(\sO_{\Pp_S(V)}(1))$, $\sigma\in H^*(S)$. It corresponds to a factor $ev_j^*(h^i\pi^*\sigma) \in H^*(\bM_{g,n}(\Pp_S(V),kf))$ in the integrand. Notice that $ev_j^*\pi^*\sigma=p^*\sigma$, and $ev_j^*h=\bar f^*(ev_j^*h')$ where $h'=c_1(\sO_{\Pp_G(U)}(1))$. In view of this, we can push forward the integrands against $[\bM_{g,n}(\Pp_S(V),kf)]^{vir_p}$ and reduce it to an integral against $[\bM_{g,n}(\Pp_G(U),kf)]^{vir_{\bar p}}$. On the other hand, $[\bM_{g,n}(\Pp_G(U),kf)]^{vir}$ and $[\bM_{g,n}(\Pp_G(U),kf)]^{vir_{\bar p}}$ are computable since we can use the localization on the Grassmannian.

This method is enough for us to establish our main theorem in the special cases of fiber classes.
\begin{lem}
Theorem \ref{main} holds if $\beta=kf$ for some $k\in\Z_{>0}$.
\end{lem}
\begin{proof}
Given $V$, $V'$ such that $c(V)=c(V')$, we can twist them by the inverse of a sufficiently ample line bundle to carry out all the above constructions at the same time. In particular, we have morphisms
\[
f, f': S\rightarrow G=Gr(r+1,N)
\]
for some $N$ such that $f^*U=V$ and $f'^*U=V'$. Notice that the Chern classes $c_i(U)$ generates the cohomology ring $H^*(G)$. Because $f^*c_i(U)=f'^*c_i(U)$, we conclude that $f, f'$ induce the same pull-back morphisms between cohomology rings, i.e., $f^*=f'^*\in Hom(H^*(G), H^*(S))$. As a result, $f_!(\sigma)=f'_!(\sigma)$ for any $\sigma\in H^*(S)$. Therefore,if we carry out all the above reductions for $V$ and $V'$ in parellel and push the integrands forward into $\bM_{g,n}(\Pp_G(U),kf)$, we get the same integrals.
\end{proof}

\subsection{Set-ups for the main algorithm}\label{section:setup}

Starting from this subsection, we work towards Theorem \ref{recursive}. We say
\[
\beta<\beta'\in \NE(X)
\]
if either $(\pi_X)_*\beta'-(\pi_X)_*\beta \in \NE(S) \backslash \{0\}$ or $\beta'-\beta \in \NE(X) \backslash \{0\}$. We are ready to state the induction hypothesis.
\begin{ind*}
We induct on the genus and the numerical curve classes. Fix a genus $g_0$ and an effective curve class $\beta_0\in \NE(\Pp(V))$. For any $g, \beta$ such that $\beta<\beta_0$ and $g\leq g_0$, assume that the formulas of invariants of the form $\langle \ldots{} \rangle^{\Pp(V)}_{g,n,\beta}$ in terms of the ones of the form $\langle \ldots{} \rangle^{S,tw,V\otimes L_{g'}^{-1}}_{g',n',\overline\beta'}$ (where $g'\leq g$)  are found.

\end{ind*}



Since fiber classes are already determined in the previous section, we may assume the following.
\begin{assu}\label{assumption:fiber}
For the rest of the section, we assume $\pi_*\beta_0 \neq 0$.
\end{assu}

\begin{goal*}
Express $\langle \psi^{k_1}\alpha_i ,\ldots{}, \psi^{k_n}\alpha_n \rangle_{g_0,n,\beta_0}^{\Pp(V)}$ in terms of Gromov--Witten invariants of the form $\langle \ldots{} \rangle_{g,n',\beta}^{\Pp(V)}$ with $\beta<\beta_0$ and $g\leq g_0$, plus the twisted invariants of $S$.
\end{goal*}


From now on, the genus $g_0$ is fixed and we stick to the line bundle $L_{g_0}$ throughout the rest of the section. Consider the Gromov--Witten invariants with target $\Pp_S(V\otimes L_{g_0}^{-1}\oplus \sO)$. To simplify notations, we may replace $V$ by $V\otimes L_{g_0}^{-1}$ just as in Assumption \ref{assumption}, which is summarized as the following. 
\begin{assu}\label{assumption:2}
For the rest of the section, in addition to the induction hypothesis and the ampleness of $\sO_{\Pp_S(V)}(1)$, we also assume $( \beta , \sO_{\Pp_S(V)}(1) ) > g_0-1$ for any $\beta\in \NE(\Pp_S(V))$ such that $\pi_*\beta\neq 0$.
\end{assu}

Recall that in the previous section, we wrote $X=\Pp(V\oplus \sO)$ and introduced $X_\infty$, $X_0$ along with other notations. We continue to use those notations. Recall we have the isomorphisms
\[
\Pp(V)\cong X_\infty\subset X, S\cong X_0\subset X.
\] 
We are going to apply virtual localization to certain invariants on $X$ and obtain relations between invariants of $\Pp(V)$ and the ones of $S$. In order to write down the localization formula, we need to fix some notations.
\begin{nota*}
\begin{enumerate}

\item To simplify notations, we write $\sO(1)=\sO_{\Pp(V)}(1)$ and $H=c_1(\sO_{\Pp(V)}(1))$.

\item Given a decorated graph $\Gamma$, for any $v\in \Gamma$, we label the markings on the component $\comp_v$ by $m_1^v,\ldots{},m_{n_v}^v$.

\item Choose a $\C$-basis $\{T_i\}$ for the cohomology ring $H^*(S)$ with dual basis written as $\{T^i\}$ (under Poincar\'e pairing).

\end{enumerate}
\end{nota*}

Besides, the following convention is used throughout the rest of the section.
\begin{conv}
Let $\iota:X_\infty\hookrightarrow X$ be the inclusion. Since we have the natural identification of the numerical curve classes $\iota_*:N_1(X_\infty)\cong N_1(X)$, we won't distinguish curve classes in $X$ and $X_\infty$ notation-wise. Furthermore, given $\beta\in N_1(X_\infty)$, since $(\beta, \sO_{X_\infty}(1))=(\iota_*\beta, \sO_X(1))$, we write $(\beta,\sO(1))$ for this intersection pairing without distinguishing whether it is evaluated on $X_\infty$ or $X$.
\end{conv}

\subsection{Lifting of insertions}\label{section:lifting}

Given a cohomology class $\alpha\in H^*(X_\infty)$, we lift $\alpha$ to $\tilde\alpha \in H^*_{\C^*}(X)$ in the way described in the following. 

$\alpha$ can always be written as 
\[
\alpha=H^e \cup \pi^*\overline{\alpha}\in H^*(X_{\infty}),
\] 
where $\overline\alpha\in H^*(S)$. Under the presentation in Proposition \ref{proposition:equivcoh}, recall $h=c_{1,T}(\sO_{\Pp(V\oplus \sO)}(1))$. We see $h$ is a lifting of $H$. Define 
\[
\tilde\alpha=h^e\cup \pi_X^*\overline{\alpha}\in H^*_{\C^*}(X),
\]
where $\pi_X:X\cong \Pp(V\oplus \sO) \rightarrow S$ is the projection. Obviously $\tilde\alpha$ restricts to $\alpha$ on $X_\infty$ and to $\lambda^e\pi^*\overline{\alpha}$ on $X_0$.

\subsection{Localization formula}\label{section:loc}

Let's introduce some notations.
\begin{nota*}
\begin{enumerate}
\item Define
\[
{\bf c}_V(x)=x^r+c_1(V)x^{r-1}+\cdots{}+c_r(V),
\]
where $r=rank(V)$.

\item When we fix a class $[e]\in \bar E(\Gamma)$, denote $v_+, v_-$ to be the two vertices in $V_{[e]}^{end}$ (whichever is arbitrary). 

\item Define $\iota_{\pm}:\Pp(V)\rightarrow X_0$ to be the projection if $p_{v_{\pm}}=0$, or $\iota_{\pm}:\Pp(V)\rightarrow X_\infty$ the identity map if $p_{v_{\pm}}=\infty$, respectively. 

\item Let $e_{\pm}$ be the edge in the class $[e]$ that contains $v_{\pm}$, respectively (they can be the same). 

\item For a vertex $v$, define $Vert(v)={\bf c}_V(\lambda),\delta_v=1$ if $p_v=0$, or $Vert(v)=H-\lambda, \delta_v=-1$ if $p_v=\infty$, respectively.
\end{enumerate}
\end{nota*}

Consider the equivariant Gromov--Witten invariant
\begin{equation}\label{eqn:comp}
\langle \psi^{k_1}\tilde\alpha_1,\ldots{},\psi^{k_n}\tilde\alpha_n \rangle^{X}_{g_0,n,\beta_0} \in \C[\lambda].
\end{equation}
Now let's write down the virtual localization formula:
\begin{align}\label{eqn:loc}
\begin{split}
& \langle \psi^{k_1}\tilde\alpha_1,\ldots{},\psi^{k_n}\tilde\alpha_n \rangle^{X}_{g_0,n,\beta_0}\\
=& \langle \psi^{k_1}\alpha_1,\ldots{},\psi^{k_n}\alpha_n \rangle^{X_\infty,tw,\sO(1)}_{g_0,n,\beta_0} +\\
& \sum\limits_{\Gamma\neq \Gamma_0}\displaystyle\frac{1}{Aut(\Gamma)}\prod\limits_{v\in V^S(\Gamma)} \sum\limits_{\{i_{[e]}\}_{[e]\in \overline E(\Gamma)-\overline E^{tail}(\Gamma)}}\langle \psi^{k_{m_1^v}}\tilde\alpha_{m_1^v}|_{X_{p_v}},\ldots{},\psi^{k_{m_{n_v}^v}}\tilde\alpha_{m_{n_v}^v}|_{X_{p_v}},\ldots{} \rangle^{X_{p_v},tw,\star}_{g_v,n,\beta_v}
\end{split}
\end{align}

A few explanations are in order:
\begin{itemize}
\item We sum over all decorated graphs with $g_\Gamma=g_0$ and $\beta_\Gamma=\beta_0$. $\Gamma_0$ is the decorated graph such that $V(\Gamma)$ consists of a single element $v_0$, $E(\Gamma)=\emptyset$, and $p_{v_0}=\infty$. 
\item We adopt the obvious convention that when $X_{p_v}=X_0$, the invariant is twisted by $V$ (with fiberwise $\C^*$ action of character $1$ for any sub-representation), and when $X_{p_v}=X_\infty$, the invariant is twisted by $\sO(1)$ (with fiberwise $\C^*$ action of character $-1$). In other words, the $\star$ symbol in the superscript at the end of the last line needs to be replaced by either $V$ or $\sO(1)$ depending on the situation.
\item Each $i_{[e]}$ in $\{i_{[e]}\}_{[e]\in \overline E(\Gamma)-\overline E^{tail}(\Gamma)}$ determines an element $T_{i_{[e]}}$ in the basis $\{T_i\}$. As suggested by the notation, they are indexed by $\overline E(\Gamma)-\overline E^{tail}(\Gamma)$. 

\item The last $\ldots{}$ sign in the twisted invariant should be inserted as follows. For any $[e]\in \overline E(\Gamma)$, we have $v_+, v_-\in V^{end}_{[e]}$ as before. Some insertions are inserted into $\langle \ldots{} \rangle_{g_{v_{\pm}},val({v_{\pm}}),\beta_{v_{\pm}}}^{X_{p_{v_{\pm}}},tw,*}$ which will be specified according to the following cases: (recall fiber integrals are treated differently. Therefore,at least one of $v_\pm$ has nontrivial degree.)
\begin{enumerate}
\item Suppose one of $v_+$ and $v_-$ is in $V^1(\Gamma)$. Say $v_-\in V^1(\Gamma)$. Then an insertion 
\[
(\iota_+)_* \left( \displaystyle\frac{\delta_{v_-}(\lambda-H)}{k_{e_-}Vert(v_-)} \cdot \frac{Edge(\Gamma,[e]) }{\delta_{v_+}(\lambda-H)/k_{e_+}-\psi} \right)
\] 
is inserted into the summand $\langle \ldots{} \rangle_{g_{v_+},val({v_+}),\beta_{v_+}}^{p_{v_+},tw,*}$. We take $\psi$ to be a formal variable under the Gysin push-forward $(\iota_+)_*$, and then it is evaluated as the $\psi$-class in the corresponding Gromov--Witten invariant. The same convention is used in the rest of the section.

\item Suppose one of $v_+$ and $v_-$ is in $V^{1,1}(\Gamma)$. Say $v_-\in V^{1,1}(\Gamma)$. Then an insertion 
\[
(\iota_+)_* \left(\dfrac{1}{Vert(v_-)} \cdot \dfrac{Edge(\Gamma,[e])}{\delta_{v_+}(\lambda-H)/k_{e_+}-\psi} \cdot \psi^{k_{m_1^{v_-}}} (\iota_-)^* \left( \tilde\alpha_{m_1^{v_-}}|X_{p_{v_-}} \right) \right)
\] 
is inserted into the summand $\langle \ldots{} \rangle_{g_{v_+},val({v_+}),\beta_{v_+}}^{p_{v_+},tw,*}$.

\item If none of the above applies, $v_+, v_-$ must all be stable vertices. In this case, an insertion 
\[
(\iota_+)_* \left( \displaystyle\frac{Edge(\Gamma,[e])T_{i_{[e]}}}{\delta_{v_+}(\lambda-H)/k_{e_+}-\psi} \right)
\]
should be placed in the summand $\langle \ldots{} \rangle_{g_{v_+},val({v_+}),\beta_{v_+}}^{p_{v_+},tw,*}$. In the meantime, an insertion 
\[
(\iota_-)_* \left( \displaystyle\frac{T^{i_{[e]}}}{\delta_{v_-}(\lambda-H)/k_{e_-}-\psi} \right)
\] 
should be placed in the $\langle \ldots{} \rangle_{g_{v_-},val({v_-}),\beta_{v_-}}^{p_{v_-},tw,*}$ summand. 

\end{enumerate}


\end{itemize}

In all the above expressions, the edge contribution $Edge(\Gamma,[e])$ can be computed as
\[
	Edge[\Gamma,[e]]=\displaystyle\frac{\displaystyle\prod\limits_{v\in V^{2}_{[e]}\bigcup V^{end}_{[e]}} Vert(v) }{ \displaystyle\prod\limits_{v\in V^{2}_{[e]}}\left( \left(\sum\limits_{e'\in E_v}\displaystyle\frac{1}{k_{e'}}\right) \delta_{v}(\lambda-H)\right) \prod\limits_{e\in [e]}\prod\limits_{m=1}^{k_e} \left(\dfrac{m}{k_e}(H-\lambda){\bf c}_V\left(H+\dfrac{m}{k_e}(\lambda-H)\right) \right) }.
\]
We also left out some special cases (for example, when both $v_+$, $v_-$ lie in $V^1(\Gamma)$ or $V^{1,1}(\Gamma)$). We leave the details to the readers as they are easy to figure out.

\begin{lem}\label{lemma:vanish}
In Equation (\ref{eqn:loc}), take $\alpha_i, k_i$ such that
\[
	\sum\limits_{i=1}^n k_i + \sum\limits_{i=1}^n deg(\alpha_i)=dim[\bM_{g_0,n}(\Pp(V),\beta_0)]^{vir}.
\]
If the curve class $\beta_0\in N_1(X)$ satisfies $\beta_0\in \NE(X)$ and $(\beta_0,\sO(1)) > g_0-1$, the left-hand side of Equation (\ref{eqn:loc}) is zero.
\end{lem}

\begin{proof}
It can be proven by dimension counting. First of all, recall
\[
	\langle \psi^{k_1}\tilde\alpha_1,\ldots{},\psi^{k_n}\tilde\alpha_n \rangle^{X}_{g_0,n,\beta_0}=\displaystyle\int_{[\bM_{g_0,n}(X,\beta_0)]^{vir}} \prod\limits_{i=1}^n \psi_i^{k_i} ev_{i}^*\alpha_i,
\]
where $[\bM_{g_0,n}(X,\beta_0)]^{vir}$ and the integration should be understood equivariantly. Notice
\begin{equation}
\begin{split}
\sum\limits_{i=1}^n k_i + \sum\limits_{i=1}^n deg(\tilde\alpha_i) &=  \sum\limits_{i=1}^n k_i + \sum\limits_{i=1}^n deg(\alpha_i) \\
&= dim[\bM_{g_0,n}(\Pp(V),\beta_0)]^{vir} \\
&= dim[\bM_{g_0,n}(X,\beta_0)]^{vir} - (1-g_0+ (\beta_0,\sO(1))) \\
&< dim[\bM_{g_0,n}(X,\beta_0)]^{vir}.
\end{split}
\end{equation}
Here we used the assumption that $( \beta_0 , \sO(1) ) > g_0-1$. Integrating a lower degree equivariant cohomology class on a higher degree equivariant homology class results in $0$, since there is no negative degree element in $H^*_{\C^*}(\{pt\})$.
\end{proof}

Recall that if $\sum\limits_{i=1}^n k_i + \sum\limits_{i=1}^n deg(\alpha_i)=dim[\bM_{g_0,n}(\Pp(V),\beta_0)]^{vir}$, by Lemma \ref{lemma:untw}, we have
\[
	\langle\psi^{k_1}\alpha_1,\ldots{},\psi^{k_n}\alpha_n\rangle_{g,n,\beta}^{X,tw,\sO(1)}=\displaystyle\frac{\langle\psi^{k_1}\alpha_1,\ldots{},\psi^{k_n}\alpha_n\rangle_{g,n,\beta}^{X}}{(-\lambda)^r}.
\]
The $(-\lambda)$ on the denominator is due to the induced action of $\C^*$ on $\sO(-1)$ that has weight $-1$ fiberwise. Therefore, under this degree condition of insertions, the leading term of the right-hand side of Equation (\ref{eqn:loc}) is in fact an untwisted invariant.

Back to the goal of this section. Note that in the last line of Equation (\ref{eqn:loc}), we sum over $\Gamma$ with $\Gamma\neq \Gamma_0$. $E(\Gamma)$ has to be nonempty. Notice that the edge component must have nontrivial numerical class. As a result, for any $v\in V(\Gamma)$, we have $\beta_v<\beta_0$. Since $g_\Gamma=g_0$, we have $g_v\leq g_0$ for all $v\in V(\Gamma)$. These observation together with Lemma \ref{lemma:vanish}, our induction is achieved by applying Equation (\ref{eqn:loc}) under the induction hypothesis and the degree condition of insertions in Lemma \ref{lemma:vanish}. In other words, Equation (\ref{eqn:loc}) expresses $\langle \psi^{k_1}\alpha_i ,\ldots{}, \psi^{k_n}\alpha_n \rangle_{g_0,n,\beta_0}^{\Pp(V)}$ in terms of Gromov--Witten invariants of the form $\langle \ldots{} \rangle_{g,n',\beta}^{\Pp(V)}$ with $\beta<\beta_0$ and $g\leq g_0$, plus the twisted invariants of $S$.

\begin{ex}
If $S$ is a point, it provides a way to compute $g=0$ Gromov--Witten invariants of $\Pp^n$ from the ones of a point. In this case, any curve class is a fiber class in the sense of section \ref{section:fiber}. Note that our localization still works for fiber integrals when $g=0$. 

Let us compute $\langle \psi^{2n-1} \rangle_{0,1,1}^{\Pp^n}$. If we use the hypergeometric $J$-function in the mirror theorem, we know the answer is
\[
\langle \psi^{2n-1} \rangle_{0,1,1}^{\Pp^n}=(-1)^n{{2n}\choose{n}}.
\]
Now we apply master space technique and consider $\langle \psi^{2n-1} \rangle_{0,1,1}^{\Pp^{n+1}}$, which is $0$ due to the virtual dimension. Let $\C^*$ act on $\Pp^{n+1}$ by sending $[x_0:\ldots{}:x_{n+1}]$ to $[\lambda x_0:\ldots{}:\lambda x_n:x_{n+1}]$. We have
\begin{align}
\begin{split}
0=\langle \psi^{2n-1} \rangle_{0,1,1}^{\Pp^{n+1}}&=\langle \psi^{2n-1} \rangle_{0,1,1}^{\Pp^{n}}\lambda^{-2}+\displaystyle\int_{\Pp^n}\dfrac{(\lambda-H)^{2n-1}(\lambda-H)}{\lambda^{n+1}(H-\lambda)}+\int_{\Pp^n}\dfrac{(-\lambda+H)^{2n-1}(-\lambda+H)}{\lambda^{n+1}(H-\lambda)}\\
&=\langle \psi^{2n-1} \rangle_{0,1,1}^{\Pp^{n}}\lambda^{-2}-2(-1)^{n}{2n-1\choose n} \lambda^{-2}.
\end{split}
\end{align}
Notice $\displaystyle 2{2n-1\choose n}={2n\choose n}$. We are done.
\end{ex}

\section{Proof of Theorem \ref{main}}\label{section:proof}

Before the proof, let's state a few lemmas.
\begin{lem}\label{lemma:tw}
Let $S$ be a smooth projective variety. Let $V_1, V_2$ be vector bundles over $S$ such that $c(V_1)=c(V_2)$. Let $\C^*$ act on $V_1, V_2$ by scaling. We have the equality of twisted invariants
\[\langle\psi^{k_1}\alpha_1,\ldots{},\psi^{k_n}\alpha_n\rangle_{g,n,\beta}^{S,tw,V_1}=\langle\psi^{k_1}\alpha_1,\ldots{},\psi^{k_n}\alpha_n\rangle_{g,n,\beta}^{S,tw,V_2} \in \C[\lambda,\lambda^{-1}].
\]
\end{lem}

It has a similar appearance as the Theorem \ref{main}. But it's all about invariants on $S$ twisted by two different vector bundles. 

\begin{proof}
Since the $\C^*$ acts on $V_i$ by scaling, we have
\[
	e_{\C^*}((V_i)_{g,n,\beta})=\lambda^r+c_1((V_i)_{g,n,\beta})\lambda^{r-1}+\cdots{}+c_r((V_i)_{g,n,\beta}).
\]
It suffices to prove that $c((V_i)_{g,n,\beta})$ depends only on the total Chern class $c(V_i)$ for $i=1,2$. But this can be seen using the Grothendieck Riemann-Roch formula
\[
ch((V_i)_{g,n,\beta})=(ft_{n+1})_*\left( ch(ev_{n+1}^*V_i) \cdot Td^\vee(\Omega_{ft_{n+1}})\right).
\]
One can get more precise formulas by following the analysis in \cite[Appendix 1]{cg} or \cite{FP}. Since $c(V_1)=c(V_2)$, we have $ch(V_1)=ch(V_2)$. And we readily have $ch(ev_{n+1}^*V_1)=ch(ev_{n+1}^*V_2)$ by functoriality. The Todd class is determined by the moduli stacks $\bM_{0,n}(X,\beta)$ and its universal family, which is independent of $V_1$ and $V_2$.
\end{proof}

\begin{lem}\label{lemma:tw2}
Let $S$ be a smooth projective variety. Let $V$ be a vector bundle with $\C^*$ acting by scaling on fibers. Then
\[
	\langle \psi^{k_1}\alpha_1,\ldots{},\psi^{k_n}\alpha_n \rangle_{g,n,\beta}^{S,tw,V}
\]
can be determined by the Chern classes $c(V)$ and untwisted invariants of the form
\[
	\langle \psi^{k'_1}\alpha'_1,\ldots{},\psi^{k'_{n'}}\alpha'_{n'} \rangle_{g',n',\beta'}^{S},
\]
where $g'\leq g$, $\beta'\leq \beta$. 
\end{lem}

Again it can be done by Grothendieck Riemann-Roch formula and we omit the details. 


The proof of \ref{main} proceeds by applying Theorem \ref{recursive} to both $\Pp(V_1)$ and $\Pp(V_2)$. But there is still one issue. Now that the corresponding twisted invariants on $S$ are identified, all the ingredients for the two cases in Theorem \ref{recursive} now agree except the Mori cones $\NE(\Pp(V_1))$ and $\NE(\Pp(V_2))$. In general, the two Mori cones can be different. However, the condition in Lemma \ref{lemma:vanish} is enough for Equation (\ref{eqn:loc}) to provide relations between invariants on projective bundles and twisted invariants. Lemma \ref{lemma:vanish} requires a weaker condition on the curve class $\beta_0$ than being effective in $\Pp(V)$.

Fix an $i\in \{1,2\}$. Recall we identify $N_1(\Pp(V_i))$ with $N_1(\Pp(V_i\oplus \sO))$ via push-forward under inclusion and we have $\NE(\Pp(V_i))\subset \NE(\Pp(V_i\oplus \sO))$ since push-forward preserves effectiveness. Also recall that by Assumption \ref{assumption} in section \ref{main}, $\sO_{\Pp(V_i)}(1)$ is ample on $\Pp(V_i)$. We prove a few lemmas under this assumption.
\begin{lem}
$\sO_{\Pp(V_i\oplus \sO)}(1)$ is nef on $\Pp(V_i\oplus \sO)$.
\end{lem}

\begin{proof}
$\Pp(V_i)$ can be realized as the zero locus of a section $s$ on the line bundle $\sO_{\Pp(V_i\oplus \sO)}(1)$. Let $C$ be an effective curve on $\Pp(V_i\oplus\sO)$. If $C\subset \Pp(V_i)$, $\sO_{\Pp(V_i\oplus \sO)}(1)$ restricts to an ample line bundle on $C$ by assumption. If otherwise, $s$ does not vanish on the whole $C$. Therefore,$\sO_{\Pp(V_i\oplus \sO)}(1)$ restricts to a line bundle of nonnegative degree as well.
\end{proof}

Recall $V_i \subset \Pp(V_i\oplus \sO)$. Let $\iota_i:S\rightarrow \Pp(V_i\oplus\sO)$ be the inclusion given by the zero section of $V_i$, and $pr_i:\Pp(V_i\oplus\sO)\rightarrow S$ the projection. Let $f_i\in\NE(\Pp(V_i\oplus\sO))$ be the class of degree $1$ curve on the fiber.
\begin{lem}
Any extremal curve class $\beta\in\NE(\Pp(V_i\oplus\sO))$ with $(pr_i)_*\beta\neq 0$ must have $\beta=(\iota_i)_*(pr_i)_*\beta$.
\end{lem}
\begin{proof}
First of all, $\beta-(\iota_i)_*(pr_i)_*\beta=kf$ for some integer $k$. If $k>0$, $\beta=kf+(\iota_i)_*(pr_i)_*\beta$ contradicts with $\beta$ being extremal. On the other hand, notice $((\iota_i)_*(pr_i)_*\beta , \sO_{\Pp(V_i\oplus\sO)}(1))=0$. We have 
\[
k = (kf , \sO_{\Pp(V_i\oplus\sO)}(1)) = (\beta - (\iota_i)_*(pr_i)_*\beta , \sO_{\Pp(V_i\oplus\sO)}(1)) = (\beta, \sO_{\Pp(V_i\oplus\sO)}(1)) \geq 0
\]
since $\sO_{\Pp(V_i\oplus\sO)}(1)$ is nef. The only case possible is $k=0$.
\end{proof}

Now $i\in \{1,2\}$ is no longer fixed, and we are going to compare the $i=1,2$ cases. Recall $c(V_1)=c(V_2)$. We have the natural identification $\Psi_{V_1\oplus\sO, V_2\oplus\sO}: N_1(\Pp(V_1\oplus\sO))\cong N_1(\Pp(V_2\oplus\sO))$ from section \ref{section:main}.
\begin{lem}
$\Psi_{V_1\oplus\sO, V_2\oplus\sO}(\NE(\Pp(V_1\oplus\sO))) = \NE(\Pp(V_2\oplus\sO))$.
\end{lem}
\begin{proof}
Certainly $\Psi_{V_1\oplus\sO, V_2\oplus\sO}(f_1)=f_2$. One can also check that given a $\overline\beta\in\NE(S)$,
\[
	\Psi_{V_1\oplus\sO, V_2\oplus\sO}((\iota_1)_*\overline\beta) = (\iota_2)_*\overline\beta.
\]
As a result, extremal rays of the corresponding Mori cones are identified under $\Psi_{V_1\oplus\sO, V_2\oplus\sO}$.
\end{proof}

Under the isomorphism $\Psi_{V_1\oplus\sO, V_2\oplus\sO}$, we will not distinguish Mori cones for $\Pp(V_1\oplus\sO)$ and $\Pp(V_2\oplus\sO)$. We are ready to apply the computation from section \ref{section:recursion} to $V=V_1$ and $V=V_2$ separately and compare the invariants. Before running the induction in section \ref{section:recursion}, the $L_g$ are chosen so that $(\beta , \sO_{\Pp(V_i)}(1)+L_g) > \text{max}\{g-1,0\}$ for both $i=1,2$. Fix $g_0\in\Z_{\geq 0}$ and $\beta_0\in\NE(\Pp(V_1))\bigcup\NE(\Pp(V_2))$. Since Lemma \ref{lemma:vanish} holds under this condition, Equation (\ref{eqn:loc}) can be applied to both $V=V_1$ and $V=V_2$ cases. Assume that invariants $\langle \ldots{} \rangle_{g,n',\beta}^{\Pp(V_i)}$ with $g\leq g_0$ and $\beta<\beta_0$ are identified as in Theorem \ref{main} for $i=1,2$. Now the lower order terms (the graph sums in Equation \eqref{eqn:loc} starting with $\sum\limits_{\Gamma\neq\Gamma_0}\displaystyle\frac{1}{Aut(\Gamma)}\cdots{}$) are identified by induction and Lemma \ref{lemma:tw}, so are the leading terms. Thus, Theorem \ref{main} is proven by induction.

\section{An application to blow-up formula}\label{section:appl}
Theorem \ref{main} can be applied to blow-ups at smooth centers to imply the following
\begin{thm}\label{blowup}
Let $Z\subset Y$ be inclusions of smooth projective varieties and $N_{Z/Y}$ be the normal bundle of $Z$. Let $\tilde Y$ be the blow-up of $Y$ at $Z$ and $E$ be the exceptional divisor. The absolute Gromov--Witten invariants of $\tilde Y$ can be determined by the absolute Gromov--Witten invariants of $Y$ and $Z$, plus the following topological data:
\begin{enumerate}
\item The cohomology rings $H^*(Y)$, $H^*(Z)$ and their pull-back maps under inclusion.
\item The Chern classes $c_i(N_{Z/Y})\in H^*(Y)$.
\end{enumerate}
\end{thm}
In fact \cite[Theorem 5.15]{HLR} plus degeneration formula already implies a similar determination of Gromov--Witten invariants that requires some additional information. To be precise, \cite[Theorem 5.15]{HLR} and degeneration formula implies
\begin{thm}\label{knownblowup}
Under the same set-up of Theorem \ref{blowup}. The absolute Gromov--Witten invariants of $\tilde Y$ can be determined by the absolute Gromov--Witten invariants of $Y$ and $Z$, plus the following data:
\begin{enumerate}
\item The cohomology rings $H^*(Y)$, $H^*(Z)$ and their pull-back maps under inclusion.
\item The Chern classes $c_i(N_{Z/Y})\in H^*(Y)$.
\item The absolute Gromov--Witten invariants of $\Pp(N_{Z/Y}\oplus \sO)$
\end{enumerate}
\end{thm}

Apparently, our Theorem \ref{main} adds to this known result by saying the invariants of $\Pp(N_{Z/Y}\oplus \sO)$ can already be determined by the invariants of $Z$ and Chern classes $c_i(N_{Z/Y})$, and thus, the requirement (c) in Theorem \ref{knownblowup} is redundant. The rest of the section is a brief explanation that \cite[Theorem 5.15]{HLR} $+$ degeneration formula implies Theorem \ref{knownblowup}.

First of all, by \cite{MP}, (3) is enough to determine all relative invariants of the pair $( \Pp(N_{Z/Y}\oplus \sO),\Pp(N_{Z/Y}) )$. Let's recall the main theorem in \cite{HLR}. In view of \cite[Definition 5.6]{HLR}, its content can be rephrased as the following.
\begin{thm}
Absolute Gromov--Witten invariants of $Y$ and (1), (2), (3) in Theorem \ref{knownblowup} are enough to determine relative Gromov--Witten invariants of the following form (relative invariants of standard weighted relative graphs in \cite{HLR}):
\[
\langle p^*\sigma_1,\ldots{},p^*\sigma_n | \mu \rangle_{g,\beta}^{\tilde Y,E},
\]
where $p:\tilde Y\rightarrow Y$ is the contraction.
\end{thm}

Although insertions are pull-backs from $Y$, we are going to see these relative invariants are enough to determine the absolute Gromov--Witten invariants of $\tilde Y$. Applying deformation to the normal cone, we form $\X=Bl_{E\times \{0\}} \tilde Y\times \mathbb A^1$. The fiber at $0$ is a union of $Y$ and $\Pp_E(\sO(-1) \oplus \sO)$ glued along $E$. Let 
\[
\iota_1:\tilde Y\rightarrow \X, \iota_2:\Pp_E(\sO(-1) \oplus \sO) \rightarrow \X
\]
be the embeddings of corresponding irreducible components to the fiber at $0$. An absolute invariant of $\tilde Y$ can thus be computed by relative invariants of the pairs $(\tilde Y,E)$ and $(\Pp_E(\sO(-1)\oplus \sO), E)$. 

During this process, one has different choices for the insertions of the relative invariants. Given $\bar \alpha\in H^*(\tilde Y)$ an insertion of the absolute invariant of $\tilde Y$, to apply degeneration formula, one needs to find a lifting $\alpha \in H^*(\X)$ whose restriction on a general fiber is $\bar \alpha$. The corresponding insertions for relative invariants of $(\tilde Y,E)$ and $(\Pp_E(\sO(-1)\oplus \sO), E)$ are $\iota_1^*\alpha$ and $\iota_2^*\alpha$, respectively. One can make an obvious choice by finding the $\alpha$ such that $\iota_1^*\alpha=\bar \alpha$ and $\iota_2^*\alpha= \pi^* (\bar\alpha|_E)$ where $\pi:\Pp_E(\sO(-1)\oplus \sO) \rightarrow E$. However, some flexibility exists. One can add to this $\alpha$ a cohomology class whose Poincar\'e dual has support in $E$. Thus, we have
\begin{lem}
Given an $\bar \alpha\in H^*(\tilde Y)$, there exists a lifting $\alpha' \in H^*(\X)$ such that
\begin{align*}
\iota_1^*\alpha&=\bar \alpha + (\iota_E)_! \sigma,\\ \iota_2^*\alpha&= \pi^* (\bar\alpha|_E) - h\pi^* \sigma,
\end{align*}
where $\iota_E:E\rightarrow \tilde Y$ is the inclusion, $\sigma\in H^*(E)$ and $h=c_1(\sO_{\Pp_E(\sO(-1)\oplus \sO)}(1))$.
\end{lem}

The goal of this section is achieved if we can prove there exists a $\sigma$ such that $\bar \alpha + (\iota_E)_! \sigma$ is the pull-back of a certain class along $p$. But notice that $\bar \alpha - p^*p_! \bar \alpha = (\iota_E)_! \sigma'$ for some $\sigma' \in H^*(E)$. Choosing $\sigma=\sigma'$ is enough. Therefore,Theorem \ref{knownblowup} is proven.

\newpage
\bibliographystyle{amsxport}
\bibliography{../universal-BIB}

\end{document}